\documentclass[12pt]{amsart}
\usepackage{amsmath,amssymb,amsbsy,amsfonts,amsthm,latexsym,
                        amsopn,amstext,amsxtra,euscript,amscd,mathrsfs,color,bm}
                        
\usepackage{float} 
\usepackage[english]{babel}
\usepackage{url}
\usepackage[colorlinks,linkcolor=blue,anchorcolor=blue,citecolor=blue,backref=page]{hyperref}
\usepackage[norefs,nocites]{refcheck}

 \usepackage[np]{numprint}

\npdecimalsign{\ensuremath{.}}

\usepackage{bibentry}

\usepackage[english]{babel}
\usepackage{mathtools}
\usepackage{todonotes}
\usepackage{url}
\usepackage[colorlinks,linkcolor=blue,anchorcolor=blue,citecolor=blue,backref=page]{hyperref}

\begin{document}

\newcommand{\commA}[2][]{\todo[#1,color=yellow]{A: #2}}
\newcommand{\commI}[2][]{\todo[#1,color=green!60]{I: #2}}
    
\newtheorem{theorem}{Theorem}
\newtheorem{lemma}[theorem]{Lemma}
\newtheorem{example}[theorem]{Example}
\newtheorem{algol}{Algorithm}
\newtheorem{corollary}[theorem]{Corollary}
\newtheorem{prop}[theorem]{Proposition}
\newtheorem{definition}[theorem]{Definition}
\newtheorem{question}[theorem]{Question}
\newtheorem{problem}[theorem]{Problem}
\newtheorem{remark}[theorem]{Remark}
\newtheorem{conjecture}[theorem]{Conjecture}

\def\xxx{\vskip5pt\hrule\vskip5pt}

\def\Cmt#1{\underline{{\sl Comments:}} {\it{#1}}}

\newcommand{\Modp}[1]{
\begin{color}{blue}
 #1\end{color}}
 
 \def\bl#1{\begin{color}{blue}#1\end{color}} 
 \def\red#1{\begin{color}{red}#1\end{color}} 


\def\cA{{\mathcal A}}
\def\cB{{\mathcal B}}
\def\cC{{\mathcal C}}
\def\cD{{\mathcal D}}
\def\cE{{\mathcal E}}
\def\cF{{\mathcal F}}
\def\cG{{\mathcal G}}
\def\cH{{\mathcal H}}
\def\cI{{\mathcal I}}
\def\cJ{{\mathcal J}}
\def\cK{{\mathcal K}}
\def\cL{{\mathcal L}}
\def\cM{{\mathcal M}}
\def\cN{{\mathcal N}}
\def\cO{{\mathcal O}}
\def\cP{{\mathcal P}}
\def\cQ{{\mathcal Q}}
\def\cR{{\mathcal R}}
\def\cS{{\mathcal S}}
\def\cT{{\mathcal T}}
\def\cU{{\mathcal U}}
\def\cV{{\mathcal V}}
\def\cW{{\mathcal W}}
\def\cX{{\mathcal X}}
\def\cY{{\mathcal Y}}
\def\cZ{{\mathcal Z}}

\def\C{\mathbb{C}}
\def\F{\mathbb{F}}
\def\K{\mathbb{K}}
\def\L{\mathbb{L}}
\def\G{\mathbb{G}}
\def\Z{\mathbb{Z}}
\def\R{\mathbb{R}}
\def\Q{\mathbb{Q}}
\def\N{\mathbb{N}}
\def\M{\textsf{M}}
\def\U{\mathbb{U}}
\def\P{\mathbb{P}}
\def\A{\mathbb{A}}
\def\fp{\mathfrak{p}}
\def\n{\mathfrak{n}}
\def\X{\mathcal{X}}
\def\x{\textrm{\bf x}}
\def\w{\textrm{\bf w}}
\def\a{\textrm{\bf a}}
\def\k{\textrm{\bf k}}
\def\ee{\textrm{\bf e}}
\def\ovQ{\overline{\Q}}
\def \Kab{\K^{\mathrm{ab}}}
\def \Qab{\Q^{\mathrm{ab}}}
\def \Qtr{\Q^{\mathrm{tr}}}
\def \Kc{\K^{\mathrm{c}}}
\def \Qc{\Q^{\mathrm{c}}}
\newcommand \rank{\operatorname{rk}}
\def\ZK{\Z_\K}
\def\ZKS{\Z_{\K,\cS}}
\def\ZKSf{\Z_{\K,\cS_f}}
\def\ZKSfG{\Z_{\K,\cS_{f,\Gamma}}}

\def\bF{\mathbf {F}}

\def\({\left(}
\def\){\right)}
\def\[{\left[}
\def\]{\right]}
\def\<{\langle}
\def\>{\rangle}

\def\gen#1{{\left\langle#1\right\rangle}}
\def\genp#1{{\left\langle#1\right\rangle}_p}
\def\genPs{{\left\langle P_1, \ldots, P_s\right\rangle}}
\def\genPsp{{\left\langle P_1, \ldots, P_s\right\rangle}_p}

\def\e{e}

\def\eq{\e_q}
\def\fh{{\mathfrak h}}

\def\lcm{{\mathrm{lcm}}\,}

\def\({\left(}
\def\){\right)}
\def\fl#1{\left\lfloor#1\right\rfloor}
\def\rf#1{\left\lceil#1\right\rceil}
\def\mand{\qquad\mbox{and}\qquad}

\def\jt{\tilde\jmath}
\def\ellmax{\ell_{\rm max}}
\def\llog{\log\log}

\def\m{{\rm m}}
\def\ch{\hat{h}}
\def\GL{{\rm GL}}
\def\Orb{\mathrm{Orb}}
\def\Per{\mathrm{Per}}
\def\Preper{\mathrm{Preper}}
\def \S{\mathcal{S}}
\def\vec#1{\mathbf{#1}}
\def\ov#1{{\overline{#1}}}
\def\Gal{{\mathrm Gal}}
\def\Sp{{\mathrm S}}
\def\tors{\mathrm{tors}}
\def\PGL{\mathrm{PGL}}
\def\wH{{\rm H}}
\def\Gm{\G_{\rm m}}

\def\house#1{{%
    \setbox0=\hbox{$#1$}
    \vrule height \dimexpr\ht0+1.4pt width .5pt depth \dp0\relax
    \vrule height \dimexpr\ht0+1.4pt width \dimexpr\wd0+2pt depth \dimexpr-\ht0-1pt\relax
    \llap{$#1$\kern1pt}
    \vrule height \dimexpr\ht0+1.4pt width .5pt depth \dp0\relax}}

\newcommand{\bfalpha}{{\boldsymbol{\alpha}}}
\newcommand{\bfomega}{{\boldsymbol{\omega}}}

\newcommand{\Ch}{{\operatorname{Ch}}}
\newcommand{\Elim}{{\operatorname{Elim}}}
\newcommand{\proj}{{\operatorname{proj}}}
\newcommand{\h}{{\operatorname{\mathrm{h}}}}
\newcommand{\ord}{\operatorname{ord}}

\newcommand{\hh}{\mathrm{h}}
\newcommand{\aff}{\mathrm{aff}}
\newcommand{\Spec}{{\operatorname{Spec}}}
\newcommand{\Res}{{\operatorname{Res}}}

\def\fA{{\mathfrak A}}
\def\fB{{\mathfrak B}}

\numberwithin{equation}{section}
\numberwithin{theorem}{section}

\title{On semigroup orbits of polynomials in subgroups}

\author[Jorge Mello]{Jorge Mello}

\address{University of New South Wales. mailing adress:\newline School of Mathematics and Statistics
UNSW Sydney 
NSW, 2052
Australia.} 

\email{j.mello@unsw.edu.au}


\begin{abstract} We study intersections of semigroup orbits in polynomial dynamics with multiplicative subgroups,  extending results of Ostafe and Shparlinski (2010). 
\end{abstract}

\maketitle
\section{Introduction}
Let $K$ be a field of charateristic $0$ and $\overline{K}$ its algebraic closure. Let $\mathcal{F}= \{ \phi_1,..., \phi_k \} \subset K[X]$ be a set of polynomials of degree at least $2$, let $x \in K$, and let 
\begin{center}$\mathcal{O}_{\mathcal{F}}(x)= \{ \phi_{i_n} \circ ... \circ \phi_{i_1}(x) | n \in \mathbb{N} , i_j =1,...,k. \}$\end{center} denote the forward orbit of $P$ under $\mathcal{F}$.
We denote the $n$-dimensional torus $\mathbb{G}^n_m$ as $(\overline{\mathbb{Q}}^*)^n$ endowed with the  group law defined by the multiplication coordinate by coordinate.

For $\mathcal{S} \subset K$ reasonably sparse and somehow unrelated to $\mathcal{F}$, it is natural to study the intersection $\mathcal{O}_{\mathcal{F}}(x) \cap \mathcal{S}$. A generalisation of this situation is to study the intersection of orbits generated by multivariate polynomials with higher dimensional algebraic varieties. This is known as the \textit{dynamical Mordell-Lang conjecture}, for which we refer \cite{BGT}. In the univariate case, when $\mathcal{S}= \mathbb{U}$ is the set of roots of unity and the initial points are defined over the cyclotomic closure $K^c:=K(\mathbb{U})$ over an algebraic number field,  Ostafe \cite{O} has proved finiteness for such points that are preperiodic for the initial polynomial.

When $k=1$ and $\mathcal{S} \subset K$ has certain multiplicative properties in the univariate case ( e.g. a finitely generated group $\Gamma \subset K^*$) Ostafe and Shparlinski \cite{OS} have provided results  for the frequency of intersections of polynomial orbits with such sets. Namely, they have proved that
\begin{center}
 $\# \{ n \leq N : f^{(n)}(x) \in \Gamma\} \leq \dfrac{(10 \log \deg f + o(1))N}{\log \log N}$, as $N \rightarrow \infty$,
\end{center} for $f \in K[X], x \in K$.

In this paper we seek to generalise results of this sort when the dynamical systems are generated as  semigroups under composition by several maps initially. Precisely, putting $\mathcal{F}_n= \{\phi_{i_n}\circ ... \circ \phi_{i_1}| 1 \leq i_j \leq k  \}$ for the $n$-level set, and supposing that $\{t_N\}_N^{+\infty}$ is a sequence of positive integers going to $\infty$  satisfying that
\begin{center}
  $\# \{ v \in  \Gamma  | v = f(u), f \in \mathcal{F}_n, n \leq N   \} \geq ck^{t_N}$
 \end{center} for each $N$, where $c>0$ is a constant,
 we prove among other results that
\begin{center}
$\# \{ v \in  \Gamma  | v = f(u), f \in \mathcal{F}_n, n \leq N   \} \leq  \exp ( \exp((10 \log d  + o(1))t_N)   $, 
\end{center} as $N \rightarrow \infty $, where $d = \max_i \deg \phi_i$. Namely, if the number of orbit points of iteration order at most $N$ falling on a finitely generated group is bigger than a multiple of the size of the complete $k$-tree of depth $t_N-1$, then such pursued number grows slower than a sequence obtained by exponentiating twice a multiple of the sequence $\{t_N\}$. 

In particular, if our conditions are satisfied with

\begin{center}
$t_N \sim\dfrac{1}{10 \log d + o(1)} \log \log \left(  \dfrac{(10 \log d + o(1)) N)}{\log \log N} \right)$, as $ N \rightarrow \infty$,
\end{center} then we can generalize and recover the results of \cite{OS} under such conditions.

In Section~\ref{sec2} we set some notations and facts about heights, orbits and finitely generated groups. In Section~\ref{sec3} we recall some  arithmetic and combinatoric results that are used to obtain results of frequency with orbits generated by a sequence of maps in Section~\ref{sec4}. In Section~\ref{sec5}  we state a necessary recent graph theory result that is used in Section~\ref{sec6}  to obtain results about the frequency of intersection of polynomial semigroup orbits with sets.

\section{Preliminar notations} \label{sec2} Let $K$ be a field of charateristic $0$ and $\overline{K}$ its algebraic closure. For $x \in \overline{\mathbb{Q}} $, the naive logarithmic height $h(x)$ is given by 
\begin{center}$ \sum_{v \in M_K} \dfrac{[K_v: \mathbb{Q}_v]}{[K:\mathbb{Q}]} \log(\max \{1, |x|_v\}$, \end{center}
where $M_K$ is the set of places of $K$, $M_K^\infty$ is the set of archimedean (infinite) places of $K$, $M_K^0$ is the set of nonarchimedean (finite) places of $K$, and for each $v \in M_K$, $|.|_v$ denotes the corresponding absolute value on $K$ whose restriction to $\mathbb{Q}$ gives the usual $v$-adic absolute value on $\mathbb{Q}$.
Also, we write $K_v$ for the completion of $K$ with respect to $|.|$, and we let $\mathbb{C}_v$ denote the completion of an algebraic closure of $K_v$. To simplify notation, we let $d_v=[K_v:\mathbb{Q}_v]/[K:\mathbb{Q}]$.
Let $\mathcal{F}= \{ \phi_1,..., \phi_k \} \subset K[X]$ be a set of polynomials of degree at least $2$, let $x \in K$, and let $\mathcal{O}_{\mathcal{F}}(x)= \{ \phi_{i_n} \circ ... \circ \phi_{i_1}(x) | n \in \mathbb{N} , i_j =1,...,k. \}$ denote the forward orbit of $P$ under $\mathcal{F}$.
We denote the $n$-dimensional torus $\mathbb{G}^n_m$ as $(\overline{\mathbb{Q}}^*)^n$ endowed with the  group law defined by the multiplication coordinate by coordinate.
\begin{definition}
 A polynomial $F \in \overline{\mathbb{Q}}[X,Y]$ is said to be special if it has a factor of the form $aX^mY^n - b$ or $aX^m - bY^n$ for some $a, b \in \overline{\mathbb{Q}}$ and $m,n \geq 0$. Otherwise we call $F$  to be non-special.
\end{definition}

\begin{definition}
 For a finitely generated group $\Gamma \subset \mathbb{G}_m^n$, we define the division group $\overline{\Gamma}$ by
\begin{center}
 $\overline{\Gamma}= \{ x \in \mathbb{G}_m^n | \exists t \in \mathbb{N}$ with $ x^t \in \Gamma \}$.
\end{center}

 \end{definition}

\begin{definition}
 For $E, \epsilon \geq 0$ and a set $\mathcal{S} \subset \mathbb{G}_m^n$, we define the sets 
 \begin{center}
  $\mathscr{B}_n(\mathcal{S}, E)= \{ x \in \mathbb{G}_m^n | \exists y, z \in \mathbb{G}_m^n $ with  $ x= yz, y \in \mathcal{S}, h(z) \leq E \}$
 \end{center} and 
 \begin{align*}
  \mathscr{C}_n(\mathcal{S}, \epsilon)= \{ x \in \mathbb{G}_m^n | \exists y, z \in \mathbb{G}_m^n \text{ with }  x= yz, y \in \mathcal{S}, h(z) \leq \epsilon(1+ h(y)) \}.
 \end{align*}

\end{definition}

We also omit the subscript $n$ for $n=1$ writing 
\begin{center}
 $\mathscr{B}(\mathcal{S}, E)= \mathscr{B}_1(\mathcal{S}, E)$ and $\mathscr{C}(\mathcal{S}, \epsilon)=\mathscr{C}_1(\mathcal{S}, \epsilon)$.
\end{center}

We also write $\mathscr{A}(K, H)$ for the set of elements in the field of height at most $H$, namely
\begin{center}
 $\mathscr{A}(K, H)= \{ x \in \overline{K}^* | h(x) \leq H  \}$.
\end{center} 
For $\mathcal{F}= \{\phi_1,..., \phi_k \}$, we set 
\begin{center}$J=\{ 1,...,k \}, \quad W= \prod_{i=1}^\infty J$, \quad and \quad $\Phi_w:=(\phi_{w_j})_{j=1}^\infty$\end{center}
 to be a sequence of polynomials from $\mathcal{F}$ for $w= (w_j)_{j=1}^\infty \in W$. 

In this situation we let \begin{center}$\Phi_w^{(n)}=\phi_{w_n} \circ ... \circ \phi_{w_1}$ with $\Phi_w^{(0)}=$Id, 
and also $\mathcal{F}_n :=\{ \Phi_w^{(n)} | w \in W \}$.\end{center}

Precisely, we consider polynomials sequences $\Phi$ $= (\phi_{i_j})_{j=1}^\infty \in \prod_{i=1}^\infty \mathcal{F}$ and $x \in \overline{K}$,
denoting $\Phi^{(n)}(x):=\phi_{i_n}(\phi_{i_{n-1}}(...(\phi_{i_1}(x)))$.

The set \begin{align*}\{ x, \Phi^{(1)}(x),  \Phi^{(2)}(x),  \Phi^{(3)}(x),... \} 
 =\{ x, \phi_{i_1}(x), \phi_{i_2}(\phi_{i_1}(x)), \phi_{i_3}(\phi_{i_2}(\phi_{i_1}(x)),... \}\end{align*} is called the forward orbit of $x$ under $\Phi$, denoted by
$\mathcal{O}_{\Phi} (x)$. 

The point $x$ is said to be $\Phi$-preperiodic if $\mathcal{O}_{\Phi} (x)$ is finite.

 For a $x \in K$, the $\mathcal{F}$-orbit of $x$ is defined as 
\begin{align*}
 \mathcal{O}_{\mathcal{F}}(x)=\{ \phi(x) | \phi \in \bigcup_{n \geq 1} \mathcal{F}_n \}= \{ \Phi_w^{(n)}(x) | n \geq 0, w \in W \} = \bigcup_{w \in W} \mathcal{O}_{\Phi_w} (x).
 \end{align*}
 The point $x$ is called preperiodic for $\mathcal{F}$ if $\mathcal{O}_{\mathcal{F}}(x)$ is finite.
 
 For $\mathcal{S} \subset K$ and an integer $N \geq 1$, we use $T_{x,\Phi}(N, \mathcal{S})$ to denote the number of $n \leq N$ with $\Phi^{(n)}(w) \in \mathcal{S}$, namely, 
 \begin{center}
  $T_{x,\Phi}(N, \mathcal{S}) = \# \{ n \leq N | \Phi^{(n)}(x) \in \mathcal{S} \}$.
 \end{center}
 For $f = \sum_{i=0}^d a_i X^i \in \overline{\mathbb{Q}}[X]$ and $K$ a field containing all the coefficients of $f$, denote the weil height of $f$ by 
 \begin{center}
  $h(f) = \sum_{v \in M_K} d_v \log(\max_i |a_i|_v)$,
 \end{center} and for the system of polynomials $\mathcal{F}= \{ \phi_1,..., \phi_k \}$, denote $h(\mathcal{F})=\max_i h({\phi_i})$. 
 
 We revisit the following bound calculated in other works, for example, \cite[Proposition 3.3]{M}.

 \begin{prop}\label{prop2.4}
  Let $\mathcal{F}= \{ \phi_1,..., \phi_k \}$ be a finite set of polynomials over $K$ with $\deg \phi_i= d_i \geq 2$, and $d:= \max_i d_i$. Then for all $n \geq 1$ and $\phi \in \mathcal{F}_n$, we have
  \begin{center}
   $h(\phi) \leq \left(\dfrac{d^n-1}{d-1}\right)h(\mathcal{F}) + d^2\left(\dfrac{d^{n-1}-1}{d-1}\right)\log 8= O(d^n(h(\mathcal{F})+1))$.
  \end{center}
 \end{prop}
 The following is an easy consequence of \cite[Corollary 2.3]{OS}.
 \begin{prop}
  Let $K$ be an number field and $\mathcal{F}= \{ \phi_1,..., \phi_k \} \subset K[X]$ a dynamical system of polynomials. Let also $g \in k[X]$ be such that $g, g \circ \phi_1,..., g\circ \phi_k$ have at least two distinct roots in $\overline{\mathbb{Q}}$.
Then, for every finitely generated subgroup $\Gamma \subset K^*, x \in \overline{\mathbb{Q}}, E >0$, we have that
 \begin{center}
  $\mathcal{O}_{\mathcal{F}}(x) \cap g^{-1}(\mathscr{B}(\overline{\Gamma},E))$
 \end{center} is finite.
 \end{prop}

\section{Some preliminar results} \label{sec3}

We define the height of $\bold{x}= (x,y) \in \mathbb{G}_m^2$ by $h(\bold{x})= h(x) + h(y)$.

For $F \in \overline{\mathbb{Q}}[X,Y]$  an absolutely irreducible polynomial of degree $d$ and height $h$, which is not special, 
we use the notation $\Delta = \deg_X F + \deg_Y F$.

For $\Gamma$ a finitely generated subgroup of $\mathbb{G}^2_m$ of rank $r >0$, we take $K$ to be the smallest number field containing all coefficients of $F$ and the group $\Gamma$, so that
\begin{center}
 $F \in K[X,Y]$ and $\Gamma \subset (K^*)^2$.
\end{center}
Letting $\mathcal{C} \subset \overline{\mathbb{Q}}^2$ be the zero set of the above polynomial, we state the following technical result.
\begin{lemma}\label{lem3.1}\cite[Lemma 4.5]{OS}
 Let $K, \Gamma, \mathcal{C}, \Delta$ and $h$ as above with $\Delta \geq 2$. Then there is a constant $c_0(K, \Gamma)$ depending only on $K$ and the generators of $\Gamma$, such that for $\zeta$ defined by
 \begin{center}
  $\zeta^{-1}= c_0(K,\Gamma)\exp (2\Delta^2)\Delta^{7r+22}(\Delta +h)(\log \Delta)^6$,
 \end{center} where $r$ is the rank of $\Gamma$, we have that
 \begin{center}
  $\# \left( \mathcal{C} \cap \mathscr{C}_2(\overline{\Gamma}, \zeta) \right) \leq \exp \left( (h+1) \exp \left( (2 + o(1)) \Delta^2 \right)   \right)$.
 \end{center}
\end{lemma}
In other side and more generally, if  $K$ is an algebraically closed field of characteristic zero, we consider polynomials $F \in K[X]$ that are not monomials.
If one denotes \begin{center}$A(n,r)= (8n)^{4n^4(n+r+1)}$, \end{center} we quote the following counting result.
\begin{lemma}\label{lem3.2}\cite[Lemma 4.7]{OS}
 Let $F \in K[X]$ be a polynomial of degree $D$ which is not a monomial and let $\Gamma \subset K^*$ be a multiplicative subgroup of rank $r$. Then
 \begin{center}
  $\# \{ (u,v) \in \Gamma^2 | F(u)=v \} < D. A(D+1, r) + D. 2^{D+1}$.
 \end{center}

\end{lemma}
We also will make use of the combinatorial statement below, which has been used and proved in a number of works.
\begin{lemma}\cite[Lemma 4.8]{OS}
 Let $2 \leq T < N/2$. For any sequence 
 \begin{center}
  $0 \leq n_1 < ... < n_T \leq N$,
 \end{center}there exists $r \leq 2N/T$ such that $n_{i+1}- n_i=r$ for at least $T(T-1)/4N$ values of $i \in \{ 1,...,T-1 \}$.
\end{lemma}
The following result for more general fields is a direct application of the previous lemma.

\begin{prop}\label{prop3.4}
 Let $K$ be an arbitrary field, $x \in K$ and let $\mathcal{S} \subset K$ be an arbitrary subset of $K$. Suppose there exist a  real number $0< \tau < 1/2$, and also $\Phi$ a sequence of polynomials contained in $\mathcal{F}= \{ \phi_1,.., \phi_k \} \subset K[X]$ such that
 \begin{center}
  $T_{x,\Phi}(N, \mathcal{S})= \tau N \geq 2$.
 \end{center} Then there exists an integer $t \leq 2 \tau^{-1}$ such that
 \begin{center}
  $\# \{ (u,v) \in \mathcal{S}^2 | \exists \psi \in \mathcal{F}_t $ with $ \psi (u)=v  \} \geq \dfrac{\tau^2N}{8}$.
 \end{center}

\end{prop}

\begin{proof}
 Letting $T:= T_{x,\Phi}(N, \mathcal{S})$, we consider all the values $1\leq n_1 <...< n_T \leq N$ such that $\Phi^{(n_i)}(x) \in \mathcal{S}, i=1,...,T-1$.
 
 From the previous lemma, there exists $t \leq 2 \tau^{-1}$ such that the number of $i=1,..., T-1$ with $n_{i+1}- n_i=t$ is at least
 \begin{center}
  $\dfrac{T(T-1)}{4N} = \dfrac{T^2}{4}\left( 1 - \dfrac{1}{T} \right) = \dfrac{\tau^2 N}{4} \left( 1 - \dfrac{1}{T} \right) \geq \dfrac{\tau^2N}{8}$.
 \end{center} Moreover, if $\mathcal{J}:= \{ 1 \leq j \leq T-1 | n_{j+1} - n_j = t \}$,  then for each $ j \in \mathcal{J}$,
 \begin{center}
  $\Phi^{(n_j)}(x) \in \mathcal{S}$ and $\Phi^{(n_{j+1})}(x)= \psi(\Phi^{(n_j)}(x)) \in \mathcal{S}$, where $\psi \in \mathcal{F}_t$.
 \end{center} and hence 
 \begin{center}
  $\# \{ (u,v) \in \mathcal{S}^2 | \psi (u)=v $ for some $\psi \in \mathcal{F}_t \} \geq \dfrac{\tau^2N}{8}$.
 \end{center}
\end{proof}

  \section{Orbits in sets} \label{sec4}
  \begin{definition}
   We say that an orbit $\mathcal{O}_{\mathcal{F}}(x)$ of an element $x \in K$ under a semigroup generated by a finite set $\mathcal{F}$ intersects a family of sets $\mathcal{S}=\{ \mathcal{S}_N \}_{N \in \mathbb{N}}$ \textit{with low frequency} if
   \begin{center}
    $\displaystyle\lim_{N \rightarrow \infty} \dfrac{\displaystyle\max_{\Phi \textit{ sequence of } \mathcal{F}} T_{x, \Phi}(N, \mathcal{S}_N)}{N}=0$
   \end{center} In the particular case that $\mathcal{S}$ is a set with $S:=\mathcal{S}_1=\mathcal{S}_2=...$ in the limit above, we say that $\mathcal{O}_{\mathcal{F}}(x)$ \textit{intersects the set $\mathcal{S}$ with low frequency}.

  \end{definition}

  Now we give a result for the frequency of intersection of orbits of semigroups of polynomials with the set $\mathscr{C}(\overline{\Gamma}, \epsilon)$ for a finitely generated subgroup $\Gamma \subset \mathbb{G}_m$.
\begin{theorem}\label{th4.2}
  Let $K$ be an number field and $\mathcal{F}= \{ \phi_1,..., \phi_k \} \subset K[X]$ a finite set of polynomials that are not monomials with  $\deg \phi_i= d_i \geq 2$, and $d = \max_i d_i$.
Suppose that, for  a finitely generated subgroup $\Gamma \subset K^*$ of rank $r$, $x \in \overline{\mathbb{Q}}$, and $\theta_N=(\log N)^{-2}(\log \log N)^{-7r/2 -12}$, we have that $\mathcal{O}_{\mathcal{F}}(x)$ intersects the  family of sets $\{ \mathscr{C}(\overline{\Gamma}, \theta_N)\}_N$ with low frequency.
Then
\begin{center}
 $\displaystyle\max_{\Phi \textit{ sequence of } \mathcal{F}} T_{x, \Phi}(N, \mathscr{C}(\overline{\Gamma}, \theta_N)) \leq \dfrac{(4 \log d  + o(1))N}{ (\log \log \log N)}$, as $N \rightarrow \infty$.
\end{center} 
\end{theorem}

\begin{proof}
 For each $\Phi$, we define $ \tau_{\Phi}$ by $\tau_\Phi =T_{x,\Phi}(N, \mathscr{C}(\overline{\Gamma}, \theta_N))/N$.

We can assume that 
\begin{equation}\label{eq4.1}
 \tau_{\Phi} \geq \dfrac{4 \log d }{\log \log \log N} \geq \dfrac{2}{N}
\end{equation} for some $\Phi$, for otherwise there is nothing to be proved.

For $N$ large enough, Proposition~\ref{prop3.4} shows that there exists 
\begin{center}
$t_{\Phi} \leq 2 \tau_{\Phi}^{-1} \leq \dfrac{\log \log \log N}{2 \log d }$
\end{center} 
such that
\begin{center}
 $\# \{ (u,v) \in \mathscr{C}(\overline{\Gamma}, \theta_N)^2 | \psi (u)=v $ for some $\psi \in \mathcal{F}_{t_{\Phi}} \} \geq \dfrac{\tau_{\Phi}^2N}{8}$.
\end{center}
For any $\psi \in \mathcal{F}_{t_{\Phi}}$, we denote by $\mathcal{C}_{\psi}$ the curve defined by the zero set of the polynomial $\psi(X) - Y=0$. Then
\begin{align*}
 \displaystyle\sum_{ \psi \in \mathcal{F}_{t_\Phi}} &\#(C_{\psi} \cap \mathscr{C}(\overline{\Gamma}, \theta_N)^2 ) \\&= \displaystyle\sum_{ \psi \in \mathcal{F}_{t_\Phi}} \# \{ (u,v) \in \mathscr{C}(\overline{\Gamma}, \theta_N)^2 | \psi (u)=v\} \\& = \# \{ (u,v) \in \mathscr{C}(\overline{\Gamma}, \theta_N)^2 | \psi (u)=v \text{ for some  }\psi \in \mathcal{F}_{t_\Phi} \} \\& \geq \dfrac{\tau_{\Phi}^2N}{8}.
 \end{align*}
The set $\{ (u,v) \in \mathscr{C}(\overline{\Gamma}, \theta_N)^2 | \psi (u)=v\}$ is the intersection of the curve $C_{\psi}$ with the set $\mathscr{C}(\overline{\Gamma}, \theta_N)^2$.

We will define a $\zeta_\Phi$ as in Lemma~\ref{lem3.1} with parameters $\Delta_\Phi=d^{t_{\Phi}} +1$ and $h=h(\mathcal{F}_{t_\Phi})$. By Proposition~\ref{prop2.4}, we have that
\begin{center}
$h \leq O(d^{t_\Phi}(h(\mathcal{F})+1))= O(\Delta_\Phi)$,
\end{center} where the referred constant does not depend on $\Phi$ satisfying (\ref{eq4.1}), but only on $\mathcal{F}$.

Moreover,
\begin{center}
 $\Delta_{t_\Phi}=d^{t_\Phi} +1 \leq (\log \log N)^{1/2} + 1$,
\end{center} and thus
\begin{align*}
 \zeta_{\Phi}^{-1}:&= \exp(2 \Delta_\Phi^2 + O(1))\Delta_\Phi^{7r + 23}(\log \Delta_\Phi)^6\\ &= O((\log N)^2 (\log \log N)^{\frac{7r + 23}{2}}(\log \log \log N)^6),
\end{align*} for $N$ sufficiently large, with the referred constant not depending on $\Phi$ satisfying (\ref{eq4.1}) again.

For our choice of $\theta_N$, we have that $\theta_N \leq \zeta_\Phi /2$ for any $N$ large enough, and so 
\begin{center}
 $\mathscr{C}(\overline{\Gamma}, \theta_N)^2 \subset \mathscr{C}_2(\overline{\Gamma} \times \overline{\Gamma}, \zeta_{\Phi})$.
\end{center} By the previous calculations, this implies that \begin{center}$\sum_{\psi \in \mathcal{F}_{t_\Phi}} \# ( \mathcal{C}_{\psi} \cap \mathscr{C}_2(\overline{\Gamma} \times \overline{\Gamma}, \zeta) \geq \dfrac{\tau_\Phi^2 N}{8}$.\end{center}

Using Lemma~\ref{lem3.1} to obtain upper bounds for the $\# ( \mathcal{C}_{\psi} \cap \mathscr{C}_2(\overline{\Gamma} \times \overline{\Gamma}, \zeta)$, knowing that $t_\Phi \leq 2 \tau_\Phi^{-1}$, we will have, as $\tau_\Phi \rightarrow 0, N \rightarrow \infty$, that
\begin{align*}
N &\leq 8\tau_{\Phi}^{-2} k^{t_\Phi} \exp (h\exp ((2 + o(1)\Delta_\Phi^2)) \\&
\leq8 \tau_{\Phi}^{-2}k^{t_\Phi} \exp (\exp( \exp ((2 \log d + o(1)t_\Phi))) \\&
\leq 8 \tau_{\Phi}^{-2}k^{t_\Phi} \exp (\exp( \exp ((4 \log d + o(1)\tau_\Phi^{-1}))) \\&
\leq  \{\exp (\exp( \exp ((4 \log d + o(1)\tau_\Phi^{-1})))\},\end{align*} and then
\begin{center}
 $\tau_\Phi \leq  \dfrac{(4 \log d  + o(1))}{\log \log \log N}$,
\end{center} and hence
\begin{center}
 $\tau_\Phi \leq \displaystyle\max_{\Phi} \tau_\Phi \leq  \dfrac{(4 \log d  + o(1))}{\log \log \log N}$
\end{center} as wanted when $N \rightarrow \infty$
.
\end{proof}
\begin{corollary}\label{cor4.3}
Under the conditions of Theorem~\ref{th4.2} we have that 
\begin{center}
$\# \{ y \in  \mathscr{C}(\overline{\Gamma}, \theta_N)  | y = f(x), f \in \mathcal{F}_n, n \leq N   \} \leq  k^N\dfrac{(4 \log d  + o(1))N}{ (\log \log \log N)} $ 
\end{center} as $N \rightarrow \infty $.
\end{corollary}
\begin{proof}
 Given $N$ very large, the set $\mathcal{F}_N$ contains $k^N$ polynomials. For each $f \in \mathcal{F}_N$, we can choose a sequence $\Phi$ of terms in $\mathcal{F}$ whose $\Phi^{(N)}= f$, obtaining $k^N$ sequences representing the elements of $\mathcal{F}_N$.
 For each sequence $\Phi$ chosen, when $N$ is large, \begin{center} $\# \{n \leq N | \Phi^{(N)}(x) \in \mathscr{C}(\overline{\Gamma}, \theta_N) \} \leq \dfrac{(4 \log d  + o(1))N}{ (\log \log \log N)}$ \end{center}
uniformly for any $\Phi$ by the previous theorem, or in other words, for each path in the $N$-tree $\mathcal{F}_N$. Since there are $k^N$ paths(polynomials, sequences) in the $n$-tree $\mathcal{F}_N$, this yields
\begin{center}
$\# \{ y \in  \mathscr{C}(\overline{\Gamma}, \theta_N)  | y = f(x), f \in \mathcal{F}_n, n \leq N   \} \leq  k^N\dfrac{(4 \log d  + o(1))N}{ (\log \log \log N)} $ 
\end{center} as $N \rightarrow \infty $.
 \end{proof} 
 \begin{theorem}\label{th4.4}
  Let $K$ be a field of charateristic zero and $\mathcal{F}= \{ \phi_1,..., \phi_k \} \subset K[X]$ a finite set of polynomials that are not monomials with  $\deg \phi_i= d_i \geq 2$ and $d = \max_i d_i$.
Then, for a finitely generated subgroup $\Gamma \subset K^*$ of rank $r, x \in K$ such that $\mathcal{O}_{\mathcal{F}}(x)$ intersects $\Gamma$ with low frequency, we have that
\begin{center}
 $\displaystyle\max_{\Phi \text{ sequence in } \mathcal{F}} T_{x, \Phi}(N, \Gamma) \leq \dfrac{(10 \log d  + o(1))N}{ (\log \log N)}$, as $N \rightarrow \infty$.
\end{center} 
\end{theorem}
\begin{proof}
 As before, we again define $\tau_\Phi = T_{x, \Phi}(N, \Gamma)/N $ and assume $\tau_{\Phi} \geq 2/N$.
Again from Proposition~\ref{prop3.4}, for $N$ large, there exists $t_\Phi \leq 2 \tau_{\Phi}^{-1}$ such that
\begin{center}
 $\dfrac{\tau_\Phi^2 N}{8} \leq \displaystyle\sum_{ \psi \in \mathcal{F}_{t_\Phi}} \# \{ (u,v) \in \Gamma^2 | \psi (u)=v\}$,
\end{center} which by Lemma~\ref{lem3.2}, as $\deg \psi \leq d^{t_\Phi}$, is upper bounded by \begin{center} $k^{2\tau_\Phi^{-1}}\left( d^{2\tau_\Phi^{-1}}A(d^{2\tau_\Phi^{-1}} +1, r) + d^{2\tau_\Phi^{-1}} 2^{d^{2\tau_\Phi^{-1}}+1}   \right)$.
\end{center}Therefore
\begin{center}
 $\dfrac{\tau_\Phi^2 N}{8} \leq k^{2\tau_\Phi^{-1}} d^{2\tau_\Phi^{-1}}A(d^{2\tau_\Phi^{-1}} +1, r) + k^{2\tau_\Phi^{-1}}d^{2\tau_\Phi^{-1}} 2^{d^{2\tau_\Phi^{-1}}+1}$.
\end{center} When $N \rightarrow \infty$ ($\tau_{\Phi} \rightarrow 0$ uniformly on $\Phi$), we have \begin{center}$N \leq 8 \tau_{\Phi}^{-2} \left( k^{2\tau_\Phi^{-1}} d^{2\tau_\Phi^{-1}}A(d^{2\tau_\Phi^{-1}} +1, r) + k^{2\tau_\Phi^{-1}}d^{2\tau_\Phi^{-1}} 2^{d^{2\tau_\Phi^{-1}}+1}   \right)$
\end{center}
bounded by
\begin{center}
 $N \leq \exp \left( \exp((10 \log d  + o(1))\tau^{-1}_{\Phi})    \right)$,
\end{center} from where the result follows.
\end{proof}
And as in Corollary~\ref{cor4.3}, the following is proven in an analogous way, working for more general fields of charateristic zero.
\begin{corollary}\label{cor4.5}
Under the conditions of Theorem~\ref{th4.4} we have that 
\begin{center}
$\# \{ y \in  \Gamma  | y = f(x), f \in \mathcal{F}_n, n \leq N   \} \leq  k^N\dfrac{(10 \log d  + o(1))N }{ (\log \log N)}$, 
\end{center} as $N \rightarrow \infty $.
\end{corollary}
\section{A graph theory result} \label{sec5}
Here we present a graph theory result of M\'{e}rai and Shparlinski \cite{MS} that will be used later in proofs.

Let $\mathcal{H}$ be a directed graph with possible multiple edges. Let $\mathcal{V}(\mathcal{H})$ be the set of vertices of $\mathcal{H}$. For $u,v \in \mathcal{V}(\mathcal{H})$, let $d(u,v)$ be the distance from $u$ to $v$, that is, the length of a shortest (directed) path from $u$ to $v$. Assume, that all the vertices have the out-degree $k\geq 1$, and the edges from all vertices are labeled by $\{1,...,k \}$.

For a word $\omega \in \{1,...,k \}^*$ over the alphabet $\{ 1,...,k\}$ and $ u \in \mathcal{V}(\mathcal{H})$, let $\omega(u) \in \mathcal{V}(\mathcal{H})$ be the end point of the walk started from $u$ and following the edges according to $\omega$.

Let us fix $u \in \mathcal{V}(\mathcal{H})$ and a subset $\mathcal{A} \subset \mathcal{V}(\mathcal{H})$. Then for words $\omega_1,..., \omega_l$ put
\begin{align*}
 L_N(u,\mathcal{A}; \omega_1,...,\omega_l)= \# \{v &\in \mathcal{V}(\mathcal{H}) : d(u,v) \leq N,\\ &d(u, \omega_i(v)) \leq N , \omega_i(v) \in \mathcal{A}, i=1,...,l  \}.
\end{align*}
To state the results, for $k,t \geq 1$, let $B(k,t)$ denote the size of the complete $k$-tree of depth $t-1$, that is
\begin{center} $  
B(k,t) = 
     \begin{cases}
       \ t
       \ &\quad\text{if} ~ k=1, \\
       \
       \
       \ \dfrac{k^t-1}{k-1}
       \ &\quad\text{otherwise} ~. \\

     \end{cases}
 $ \end{center}
 
 \begin{lemma}\label{lem5.1}
  Let $u \in \mathcal{V}(\mathcal{H})$, and $t,l \geq 1$ be fixed. If $\mathcal{A} \subset \mathcal{V}(\mathcal{H})$ is a subset of vertices with
  \begin{align*}
   \#\{ v \in \mathcal{A} : d&(u,v) \leq N \}\\ & \geq \max \left\{ 3B(k,t), \dfrac{3l}{t}\# \{v \in \mathcal{V}(\mathcal{H}) : d(u,v) \leq N \} \right\},
  \end{align*}
then there exist words $\omega_1,...,\omega_l \in \{ 1,...,k \}^*$ of length at most $t$ such that
\begin{center}
 $L_N(u,\mathcal{A}; \omega_1,...,\omega_l) \gg  \dfrac{t}{B(k,t)^{l+1}} \# \{v \in \mathcal{V}(\mathcal{H}) : d(u,v) \leq N \}$,
\end{center} where the implied constant depend only on $l$.

 \end{lemma}
\section{more results of orbits in sets} \label{sec6}

\begin{theorem}
 Let $K$ be a field of charateristic zero and $\mathcal{F}= \{ \phi_1,..., \phi_k \}\newline  \subset K[X]$ a finite set of polynomials that are not monomials, with  $\deg \phi_i= d_i \geq 2$ and $d = \max_i d_i$. 
 Suppose that  $\Gamma \subset K^*$ is a finitely generated subgroup of rank $r$, and $ u \in K$.
 Let also $t,l \geq 1$ be  integers such that $t \geq 3l$ and
 $\# \{ v \in  \Gamma  | v = f(u), f \in \mathcal{F}_n, n \leq N   \} \geq 3B(k,t)$.
Then
\begin{align*}
 \# \{ v \in  \Gamma  | v = f(u), f \in \mathcal{F}_n, n \leq N   \} \ll_l \dfrac{B(k,t)^{l+1}}{t}(d^t A(d^t+1,r) +d^t2^{d^t +1}).
\end{align*}

\end{theorem}
\begin{proof}
 We consider the directed graph with the elements of $\Gamma$ as vertices, and edges $(x,\phi_i(x))$ for $i=1,...,k$ and $x \in \Gamma$. With the notation of Section~\ref{sec5} and Lemma~\ref{lem5.1}, we let $\Gamma$ take the place of $\mathcal{H}$ and $\mathcal{A}$.
 By hypothesis, $l \leq t/3$ and $\#\{ v \in \Gamma, d(u,v) \leq N \} \geq 3B(k,t)$. From Lemma~\ref{lem5.1}, there exist words $\omega_1,...,\omega_l \in \{ 1,...,k\}^*$ of length at most $t$, and therefore degree at most $d^t$, such that
 \begin{equation}\label{eq6.1}
  L_N(u, \Gamma; \omega_1,...,\omega_l) \gg_l  \dfrac{t}{B(k,t)^{l+1}} \# \{v \in \mathcal{V}(\Gamma) : d(u,v) \leq N \}.
 \end{equation} By Lemma~\ref{lem3.2}, we compute
  \begin{align*}
  L_N&(u, \Gamma; \omega_1,...,\omega_l)\\&=\# \{v \in \mathcal{V}(\Gamma) : d(u,v),d(u, \omega_i(v)) \leq N , \omega_i(v) \in \Gamma, i=1,...,l  \}\\
  &\leq \displaystyle\sum_{i \leq l} \# \{v \in \mathcal{V}(\Gamma) : d(u,v),d(u, \omega_i(v)) \leq N , \omega_i(v) \in \Gamma  \}\\
 & \leq \displaystyle\sum_{i \leq l} \# \{(x,y) \in \Gamma^2 : y= \omega_i(x) \} \\
 &\leq \displaystyle\sum_{i \leq l} (d^t A(d^t+1,r) +d^t2^{d^t +1})\\
 &=l(d^t A(d^t+1,r) +d^t2^{d^t +1}).
 \end{align*}Gathering this with (\ref{eq6.1}),  we conclude that
 \begin{center}
  $\# \{v \in \mathcal{V}(\Gamma) : d(u,v) \leq N \} \ll_l \dfrac{B(k,t)^{l+1}l}{t}(d^t A(d^t+1,r) +d^t2^{d^t +1})$,
 \end{center} as desired.
\end{proof}
\begin{corollary}\label{cor6.2}
 Let $K$ be a field of charateristic zero and $\mathcal{F}= \{ \phi_1,..., \phi_k \}\newline  \subset K[X]$ a finite set of polynomials that are not monomials, with  $\deg \phi_i= d_i \geq 2$ and $d = \max_i d_i$. 
 Suppose that  $\Gamma \subset K^*$ is a finitely generated subgroup of rank $r$,  $ u \in K$, and $\{t_N\}_N$ is a sequence of positive integers that goes to $\infty$ as $N \rightarrow \infty$ and that satisfies
 \begin{center}
  $\# \{ v \in  \Gamma  | v = f(u), f \in \mathcal{F}_n, n \leq N   \} \geq 3B(k,t_N)$
 \end{center} for each $N$.
 Then
\begin{center}
$\# \{ v \in  \Gamma  | v = f(u), f \in \mathcal{F}_n, n \leq N   \} \leq  \exp \left( \exp((10 \log d  + o(1))t_N)    \right)$, 
\end{center} as $N \rightarrow \infty $.
 \end{corollary}
 \begin{proof}
  In the proof of the previous result, we can choose $l \geq 1$  an arbitrary integer and $N$ big enough so that $t_N \geq 3l$. For each of these $t_N$'s, we can apply the previous theorem, obtaining that
  \begin{align*}
 \# \{ v \in  \Gamma  | v = f(u), f \in \mathcal{F}_n, n \leq N   \} \ll_l \dfrac{B(k,t_N)^{l+1}}{t_N}(d^{t_N} A(d^{t_N}+1,r) +d^{t_N}2^{d^{t_N} +1}).
\end{align*} Moreover, since $t_N \rightarrow \infty$ as $N \rightarrow \infty$, it yields \begin{align*}B(k,t_N)^{l+1}(d^{t_N} A(d^{t_N}+1,r) +d^{t_N}2^{d^{t_N} +1}) =\exp \left( \exp((10 \log d  + o(1))t_N)    \right),\end{align*}
from where the result follows.
 \end{proof}
 \begin{remark}
 If the hypothesis of Corollary~\ref{cor6.2} are satisfied with \begin{center}$t_N \sim\dfrac{1}{10 \log d + o(1)} \log \log \left(  \dfrac{(10 \log d + o(1)) N)}{\log \log N} \right)$, as $ N \rightarrow \infty$, \end{center}
 then we recover and generalize Corollary~\ref{cor4.5}, as well as Theorem 3.1 of \cite{OS}, under our referred conditions.
\end{remark} 
\begin{theorem}\label{th6.4}
 Let $K$ be an algebraic number field and $\mathcal{F}= \{ \phi_1,..., \phi_k \}\newline  \subset K[X]$ a finite set of polynomials that are not monomials, with  $\deg \phi_i= d_i \geq 2$ and $d = \max_i d_i$. 
 Suppose that  $\Gamma \subset K^*$ is a finitely generated subgroup of rank $r$,  $ u \in K$, and $\{t_N\}_N$ is a sequence of positive integers that goes to $\infty$ as $N \rightarrow \infty$ and that satisfies
 \begin{center}
  $\# \{ v \in  \mathscr{C}(\overline{\Gamma}, \theta_N) | v = f(u), f \in \mathcal{F}_n, n \leq N   \} \geq 3B(k,t_N)$
 \end{center} for each $N$, whith $\theta_N \leq \left( \exp(d^{-2t_N})d^{N(-7r-24)}\right)$.
 Then
\begin{align*}
\# \{ v \in  \mathscr{C}(\overline{\Gamma}, \theta_N)  | v = f(u), f \in \mathcal{F}_n, n \leq N   \} \leq  \exp (\exp( \exp ((4 \log d + o(1))t_N))), 
\end{align*} as $N \rightarrow \infty $.
\end{theorem}
\begin{proof}
 We consider the directed graph with the elements of $\mathscr{C}(\overline{\Gamma}, \theta)$ as vertices, and edges $(x,\phi_i(x))$ for $i=1,...,k$ and $x \in \mathscr{C}(\overline{\Gamma}, \theta)$. With the notation of Section~\ref{sec5} and Lemma~\ref{lem5.1}, we let $\mathscr{C}(\overline{\Gamma}, \theta)$ take the place of $\mathcal{H}$ and $\mathcal{A}$.
 By hypothesis, we can choose $l \geq 1$  an arbitrary integer and $N$ big enough so that $t_N \geq 3l$ and $\#\{ v \in \mathscr{C}(\overline{\Gamma}, \theta), d(u,v) \leq N \} \geq 3B(k,t_N)$. From Lemma~\ref{lem5.1}, for each $N$,  there exist words $\omega_1,...,\omega_l \in \{ 1,...,k\}^*$ of length at most $t_N$, and therefore degree at most $d^{t_N}$, such that
 \begin{equation}\label{eq6.2}
  L_N(u, \mathscr{C}(\overline{\Gamma}, \theta); \omega_1,...,\omega_l) \gg_l  \dfrac{t_N}{B(k,t_N)^{l+1}} \# \{v \in \mathcal{V}(\mathscr{C}(\overline{\Gamma}, \theta)) : d(u,v) \leq N \}.
 \end{equation} Putting $\Delta_{t_N}=d^{t_N} +1$ and $h_N= h(\mathcal{F}_{t_N})$, we have $h=O(\Delta_{t_N})$ by Proposition~\ref{prop2.4}.
 Defining $\zeta_N$ as in Lemma~\ref{lem3.1} with parameters $h_N, \Delta_{t_N}$ we have that
 \begin{align*}\zeta^{-1}= \exp(2 \Delta_{t_N}^2 + O(1))\Delta_{t_N}^{7r + 23}(\log \Delta_{t_N})^6
 =O(\exp(d^{2t_N})d^{t_N(7r+23)}(t_N\log d)^6).
  \end{align*}
  As $\theta_N\leq \zeta_N/2= O\left(\left(\exp(d^{2t_N})d^{t_N(7r+23)}(t_N\log d)^6\right)^{-1}\right)$, for $N$ large enough, it is true  that
  \begin{center}
 $\mathscr{C}(\overline{\Gamma}, \theta_N)^2 \subset \mathscr{C}_2(\overline{\Gamma} \times \overline{\Gamma}, \zeta)$.
\end{center}By Lemma~\ref{lem3.1}, we compute
 \begin{align*}
  L_N&(u, \mathscr{C}(\overline{\Gamma}, \theta_N); \omega_1,...,\omega_l)\\&=\# \{v \in \mathcal{V}(\mathscr{C}(\overline{\Gamma}, \theta_N)) : d(u,v),d(u, \omega_i(v)) \leq N , \omega_i(v) \in \mathscr{C}(\overline{\Gamma}, \theta), i=1,...,l  \}\\
  &\leq \displaystyle\sum_{i \leq l} \# \{v \in \mathcal{V}(\mathscr{C}(\overline{\Gamma}, \theta_N)) : d(u,v),d(u, \omega_i(v)) \leq N , \omega_i(v) \in \mathscr{C}(\overline{\Gamma}, \theta_N)  \}\\
 & \leq \displaystyle\sum_{i \leq l} \# \{(x,y) \in \mathscr{C}(\overline{\Gamma}, \theta_N)^2 : y= \omega_i(x) \} \\
 & \leq \displaystyle\sum_{i \leq l} \# \{(x,y) \in \mathscr{C}_2(\overline{\Gamma} \times \overline{\Gamma}, \zeta_N) : y= \omega_i(x) \}\\
 & \leq l \exp ((h_N+1)\exp ((2 + o(1))\Delta_{t_N}^2)). 
 \end{align*} Gathering this with (\ref{eq6.2}), it follows that
 \begin{align*}
  &\# \{v \in \mathcal{V}(\mathscr{C}(\overline{\Gamma}, \theta_N)) : d(u,v) \leq N \}\\& \ll_l \dfrac{B(k,t_N)^{l+1}l}{t_N}\exp ((h_N+1)\exp ((2 + o(1))\Delta_{t_N}^2))\\
  & \leq  \dfrac{\exp (\exp( \exp ((4 \log d + o(1))t_N)))}{t_N},
 \end{align*}as we wanted to show.

\end{proof}
\begin{remark}
 If the hypothesis of  Theorem~\ref{th6.4} are satisfied with \begin{center}$t_N \sim\dfrac{1}{4 \log d + o(1)} \log \log \log \left(  \dfrac{(4 \log d + o(1)) N)}{\log \log \log N} \right)$, as $ N \rightarrow \infty$, \end{center}
 then we recover and generalize Corollary~\ref{cor4.3}, as well as Theorem 2.4 of \cite{OS}, under our referred conditions.
\end{remark}

\end{document}